\documentclass[12pt]{amsart}
\usepackage{amscd,amsmath,amsthm,amssymb}
\usepackage{color}
\usepackage{amsfonts,amssymb,amscd,amsmath,enumerate,verbatim}
\usepackage{lineno}
\usepackage{url}
 \def\NZQ{\mathbb}               % the font for N,Z,Q,R,C
 \def\NN{{\NZQ N}}
 
 \def\ZZ{{\NZQ Z}}

 \def\AA{{\NZQ A}}

 \def\frk{\mathfrak}               % font for "Fraktur"

 \def\mm{{\frk m}}
 
 \def\nn{{\frk n}}
 \def\ab{{\mathbf a}}

 \def\opn#1#2{\def#1{\operatorname{#2}}} % to make operators
 \opn\chara{char} \opn\length{\ell} \opn\pd{pd} \opn\rk{rk}
 \opn\projdim{proj\,dim} \opn\injdim{inj\,dim} \opn\rank{rank}
 \opn\depth{depth} \opn\grade{grade} \opn\height{height}
 \opn\embdim{emb\,dim} \opn\codim{codim}
 
 \opn\Tr{Tr} \opn\bigrank{big\,rank}
 \opn\superheight{superheight}\opn\lcm{lcm}
 \opn\trdeg{tr\,deg}%\emph{
 \opn\reg{reg} \opn\lreg{lreg} \opn\ini{in} \opn\lpd{lpd}
 \opn\size{size} \opn\sdepth{sdepth}
 \opn\link{link}\opn\fdepth{fdepth}\opn\lex{lex}
\opn\type{t}
 \opn\div{div} \opn\Div{Div} \opn\cl{cl} \opn\Cl{Cl}
 \opn\Spec{Spec} \opn\Supp{Supp} \opn\supp{supp} \opn\Sing{Sing}
 \opn\Ass{Ass} \opn\Min{Min}\opn\Mon{Mon}
 \opn\Ann{Ann} \opn\Rad{Rad} \opn\Soc{Soc}
 \opn\Im{Im} \opn\Ker{Ker} \opn\Coker{Coker} \opn\Am{Am}
 \opn\Hom{Hom} \opn\Tor{Tor} \opn\Ext{Ext} \opn\End{End}
 \opn\Aut{Aut} \opn\id{id}
 
 \opn\nat{nat}
 \opn\pff{pf}%   \pf exists already
 \opn\Pf{Pf} \opn\GL{GL} \opn\SL{SL} \opn\mod{mod} \opn\ord{ord}
 \opn\Gin{Gin} \opn\Hilb{Hilb}\opn\sort{sort}
 \opn\aff{aff} \opn
 \con{conv} \opn\relint{relint} \opn\st{st}
 \opn\lk{lk} \opn\cn{cn} \opn\core{core} \opn\vol{vol}  \opn\inp{inp} \opn\nilpot{nilpot}
 \opn\link{link} \opn\star{star}\opn\lex{lex}\opn\set{set}
 \opn\width{wd}
\opn\Ap{Ap}
\opn\mult{m}
 \opn\gr{gr}
 
 \def\pot#1#2{#1[\kern-0.28ex[#2]\kern-0.28ex]}

 \opn\dirlim{\underrightarrow{\lim}}
 \opn\inivlim{\underleftarrow{\lim}}
 \let\to=\rightarrow
 
 \def\Implies{\ifmmode\Longrightarrow \else
         \unskip${}\Longrightarrow{}$\ignorespaces\fi}
 \def\implies{\ifmmode\Rightarrow \else
         \unskip${}\Rightarrow{}$\ignorespaces\fi}
 \def\iff{\ifmmode\Longleftrightarrow \else
         \unskip${}\Longleftrightarrow{}$\ignorespaces\fi}

 \let\:=\colon
 \newtheorem{Theorem}{Theorem}[section]
 \newtheorem{Lemma}[Theorem]{Lemma}
 \newtheorem{Corollary}[Theorem]{Corollary}
 \newtheorem{Proposition}[Theorem]{Proposition}
 \newtheorem{Remark}[Theorem]{Remark}
 
 \newtheorem{Example}[Theorem]{Example}

 \newtheorem{Problem}[Theorem]{Problem}
 \newtheorem{Conjecture}[Theorem]{Conjecture}
 \newtheorem{Question}[Theorem]{Question}
 \let\epsilon\varepsilon
 \let\kappa=\varkappa
 \def\qed{\ifhmode\textqed\fi
       \ifmmode\ifinner\quad\qedsymbol\else\dispqed\fi\fi}
 \def\textqed{\unskip\nobreak\penalty50
        \hskip2em\hbox{}\nobreak\hfil\qedsymbol
        \parfillskip=0pt \finalhyphendemerits=0}
 \def\dispqed{\rlap{\qquad\qedsymbol}}
 
 \opn\dis{dis}
 \def\pnt{{\raise0.5mm\hbox{\large\bf.}}}
 
 \opn\Lex{Lex}

\begin{document}
%\linenumbers
\title {Betti numbers for numerical semigroup rings}

 \author {Dumitru I.\ Stamate}

\address{Dumitru I. Stamate, ICUB/Faculty of Mathematics and Computer Science, University of Bucharest, Str. Academiei 14, Bucharest, Romania }
\email{dumitru.stamate@fmi.unibuc.ro}

\dedicatory{}

\begin{abstract}
We survey results related to the magnitude of the Betti numbers of numerical semigroup rings and of their tangent cones.
\end{abstract}

\thanks{}
%   13-02    	Research exposition (monographs, survey articles)
% 	13D02   	Syzygies, resolutions, complexes: in Commutative algebra
% 	13A30   	Associated graded rings of ideals (Rees ring, form ring), analytic spread and related topics
% 	16S36   	Ordinary and skew polynomial rings and semigroup rings

% 	13H10   	Special types (Cohen-Macaulay, Gorenstein, Buchsbaum, etc.)
% 	13P10   	Gröbner bases; other bases for ideals and modules (e.g., Janet and border bases)


\maketitle
\setcounter{tocdepth}{1}
%\tableofcontents 
\section{Introduction}

A numerical semigroup $H$ is a subset of the set of  nonnegative integers $\NN$, closed under addition, containing $0$ (i.e. a monoid) and such that $|\NN\setminus H| <\infty$. The latter condition may be replaced by having $\gcd(h:h\in H)=1$.
The largest integer not in the numerical semigroup  $H$ is called its  Frobenius number, which we denote $F(H)$.

Given the positive integers $a_1, \dots, a_n$, the monoid they generate is
$$
\langle a_1, \dots, a_n \rangle= \left\{\sum_{i=1}^n k_i a_i: k_i \in \NN, i=1, \dots, n \right\}.
$$ 
Let $d=\gcd(a_1, \dots, a_n)$. Since the semigroups $\langle a_1, \dots, a_n\rangle$ and $\langle a_1/d, \dots, a_n/d\rangle$ 
are isomorphic, it is clear that the study of submonoids of $\NN$ reduces to the study of numerical semigroups.
It is easy to check that any monoid $H\subset \NN$ is finitely generated and that it has a unique minimal generating set that we denote  $G(H)$.  The embedding dimension of $H$ is defined as $\embdim(H)=|G(H)|$, the multiplicity of $H$ is $\mult (H)=\min G(H)$ and its width    is $\width (H)=\max G(H) - \min G(H)$.

Unless otherwise stated, throughout this paper $H$ denotes a numerical semigroup and  any semigroup is assumed to be a numerical semigroup.
 
Let $K$ be any field. The semigroup ring $K[H]$ is the $K$-subalgebra of the polynomial ring $K[t]$ generated by the monomials $t^h$ where $h\in H$. Assume $G(H)=\{a_1, \dots, a_n\}$. We consider the polynomial ring $S=K[x_1,\dots, x_n]$ endowed with the nonstandard grading induced by $H$, namely $\deg x_i= a_i$ for all $i$. Then the $K$-algebra map $\varphi:S \to K[H]$ letting
$\varphi(x_i)=t^{a_i}$ is a graded map. Its kernel $I_H=\ker \varphi$ is also called the toric ideal of $H$ since it is the ideal of relations of the toric algebra $K[H]$. More generally, any $1$-dimensional toric  ring can be viewed as a semigroup ring $K[H]$ with $H$ a numerical semigroup.

The elements of any minimal set of (homogeneous) generators of $I_H$ are the first syzygies of $K[H]$, and their count is the first Betti number of $K[H]$, denoted  $\beta_1(K[H])$. A minimal set of relations among the  first syzygies consists of what are called the second syzygies of $K[H]$, which are counted by $\beta_2(K[H])$. The process continues, and  it  produces  nontrivial syzygies until we reach the projective dimension of $K[H]$. 
The Betti sequence of $K[H]$ is the list $(\beta_0(K[H]), \beta_1(K[H]), \dots )$, where $\beta_0(K[H])=1$. The last nonzero Betti number, namely $\beta_{n-1}(K[H])$ is called the (Cohen-Macaulay) type of $K[H]$ (or of the semigroup $H$, see also Section \ref{sec:type}).

An algebra that is closely related to $K[H]$ and is relevant to our study  is its associated graded ring with respect to the maximal ideal $\mm=(t^h:  h\in H\setminus\{0\})K[H]$, i.e. $\gr_\mm K[H]= \oplus_{i\geq 0}  \mm^i/\mm^{i+1}$ .
 This algebra is also called the tangent cone of $K[H]$ (or of $H$) in resemblance with the geometric origin of the concept, see \cite[Chapter 5]{Eis}.

Minimal free resolutions of modules and their attached invariants are a classical and nevertheless constant source of questions
in algebraic geometry and in commutative algebra, see \cite{Eis-syz}, \cite{Peeva}. In this survey we collect some of the results and questions pertaining to these topics when the modules are $K[H]$ or $\gr_\mm K[H]$ for a numerical semigroup $H$. It is clear that  general results may be also applied to this setting, but on the other hand new tools become available and we get sharper statements when we focus on the type of rings we mentioned.

\medskip

We briefly outline the structure of this paper. In Section \ref{sec:algebra} we discuss arithmetic consequences (and at times characterizations) of the Gorenstein, complete intersection or Cohen-Macaulay property for $K[H]$ and $\gr_\mm K[H]$.
%In Section \ref{sec:exact} 
We present how to start the resolution of these algebras, namely how to (algorithmically) compute the toric ideal $I_H$ and its ideal of initial forms $I_H^*$.  Exact, but somehow opaque, formulas for the Betti numbers of $K[H]$ are given
in terms of topological and combinatorial data encoded in the squarefree divisor complexes of the elements in the semigroup, introduced in \cite{BH-semi}.

Fr\"oberg, Gottlieb and H\"aggkvist \cite{FGH} defined the type of the semigroup $H$ as the cardinality of its set of its pseudo-Frobenius numbers: $PF(H)=\{ x \in \ZZ \setminus H: x+h \in H \text{ for all } 0\neq h \in H \}$. 
In Section \ref{sec:type} we give a detailed proof of the fact that this notion of type coincides with the Cohen-Macaulay type of the semigroup ring $K[H]$. In higher embedding dimension the type is also unbounded, but we present  estimates for it when $H$  is in certain families of semigroups.

A recent result of Vu \cite{Vu}  is that if we bound the width of the semigroup, then the Betti numbers of $K[H]$ are bounded, too.
This was extended to the Betti numbers of $\gr_\mm K[H]$ by Herzog and the present author in \cite{HS}. 
These are  consequences of an eventual periodic behaviour of the Betti sequence of the toric ring and its tangent cone for semigroups in the same shifted family. We collect in Section \ref{sec:shifts} the most important periodicity-like results that have been spotted so far in shifted families of semigroups.
There is a growing interest in this topic, due to possible applications in faster implementations of the known algorithms for computing invariants of $K[H]$ or $H$, see \cite{ConawayAll} and \cite{ONeill-Pelayo-Apery}.

We say that the semigroup $H$ is obtained from the semigroup $L$  by a simple gluing if $H=\langle cL, \ell \rangle$ with $c, \ell >1$  coprime and $\ell \in L\setminus G(L)$. In Section \ref{sec:betti-gluings} we explain how the Betti numbers of $K[L]$ and $K[H]$ are related; in particular they have the same type.

\medskip

In Section \ref{sec:firstexamples} we discuss two families of semigroups  for whom the whole resolution of the associated toric ring is known: the complete intersections and those generated by an arithmetic sequence.
It has been known since Herzog's \cite{He-semi} that at most three binomials suffice to generate $I_H$ when $H$ is $3$-generated. On the other hand,  it is possible in this embedding dimension that $I_H^*$ have as many generators as one wants, see Section 
\ref{sec-atmost3}.

Bresinsky \cite{Bres}  produced the first examples  of $4$-generated semigroups where $\mu(I_H)$ is arbitrarily large.
In Section \ref{sec-edim4} we give a detailed proof of the computation of the Betti sequence of $K[H]$ and $\gr_\mm K[H]$ when 
$H$ is of Bresinsky's type, and also in a related family found by Arslan \cite{Arslan}.

For several families of $4$-generated semigroups the resolution of their toric ring is known, and we present these in Section \ref{sec:4continued}. Namely, when $H$ is  symmetric, pseudosymmetric or almost symmetric, or when it is generated by an almost arithmetic sequence.  It is still obscure and not yet completed (or even started) the list of possible Betti sequences of $\gr_\mm K[H]$ when $H$ is in one of the said families.

For the  background and the undefined terminology from  commutative algebra we refer to the monographs \cite{Eis}, \cite{BH}, and for an introduction to numerical semigroups   to \cite{NS-book} and \cite{Ramirez}.  The lecture notes of Fr\"oberg \cite{Froberg-Porto} from the conference in Porto in 2008 may also serve as an introduction to the topics we present.
Resolutions and toric ideals are rarely computed by hand, and we are happy that software like Singular \cite{Sing}, Macaulay2 \cite{M2} or CoCoA \cite{Cocoa} exists. In Singular, the package \texttt{toric.lib} implements several algorithms for computing toric ideals, which are particularly efficient for numerical semigroup rings since these are not standard graded. Presentations for numerical semigroups can also be computed via the \texttt{numericalsgps} package \cite{Num-semigroup} in GAP \cite{GAP} where many other semigroup routines are to be found.

Numerical semigroups occur in various  braches of mathematics, ranging from the study of singularities, number theory to coding theory. We hope this survey would be on interest to a larger audience, especially since there   is a growing number of recent publications dealing with resolutions or Betti numbers  for numerical semigroup rings.
The  results are scattered in various places, as witnessed by the list of references.
We tried to be comprehensive, but the outcome is of course limited by our  knowledge.

For most of the results we present  we preferred not to include a full proof, but rather point to one, if available. 
A lot of examples are included, and it is here where we insisted on giving details,  at places fixing some gaps in the literature.

\section{Algebraic warm-up}
\label{sec:algebra}

In this section we give some basic algebraic facts about semigroup rings, their minimal resolution, and we recall some terminology.

Let $K$ be any field, $H$   a numerical semigroup minimally generated by $a_1, \dots, a_n$ and $S=K[x_1, \dots, x_n]$   the polynomial ring over $K$ in the indeterminates $x_1, \dots, x_n$.  
On $S$ we consider the nonstandard grading given by the semigroup $H$ by letting $\deg x_i= a_i$ for $i=1, \dots, n$.
We denote $I_H$ the kernel of the $K$-algebra map $\phi:S \to K[H]$ letting $\phi(x_i)=t^{a_i}$  for $i=1, \dots, n$.
The algebra $K[H] \cong S/I_H$ is a $1$-dimensional domain,  hence by the Auslander--Buchsbaum formula  (\cite[Theorem 1.3.3]{BH}) we have that
$$
\projdim_S K[H]=\height I_H=n -1.
$$ 

Let $f_1,\dots, f_r$ be a minimal system of generators for $I_H$, which are homogeneous with respect to the grading on $S$ induced by the semigroup $H$. The relations among them are captured by the kernel of the $S$-linear map $\varphi_1:F_1=\bigoplus_{i=1}^r Se_i \to S$ letting $\varphi(e_i)= f_i$ for $i=1, \dots, r$. 
To make this map homogeneous we assign $\deg (e_i)=\deg (f_i)$ for $i=1, \dots, r$.
 
This process continues and we obtain a chain complex of free $S$-modules of finite rank
\begin{equation*}
\label{eq:f}
\mathbb{F} : \quad  0\to F_{n-1} \stackrel{\varphi_{n-1}}{\rightarrow} F_{n-2} \to \dots \to F_1 \stackrel{\varphi_{1}}{\rightarrow} F_0=S\to 0,
\end{equation*}
which is exact in positive  homological degrees, it has $H_0(\mathbb{F})\cong S/I_H$, and the maps have the property that
$\varphi_{i}(F_i) \subset \nn F_{i-1}$ for all $i=1, \dots, n-1$. Here $\nn$ denotes the maximal homogeneous ideal of $S$.
One says that $\mathbb{F}$ is a minimal free $S$-resolution of $S/I_H$. Such an $\mathbb{F}$ is unique (only) up to
an isomorphism of chain complexes,
 hence we can define the $i^{th}$ Betti number of $K[H]$ as 
$$\beta_i(K[H])=\rank F_i=\dim_K \Tor_i^S(K[H],K)  \text{ for all } i,$$ 
and  this number does not depend on the chosen minimal free resolution of $K[H]$.
The Betti sequence of $K[H]$ is the list $(\beta_0(K[H]), \beta_1(K[H]), \dots )$. Clearly, $\beta_0(K[H])=1$ and $\beta_1(K[H])=\mu(I_H)$ the minimal number of generators for $I_H$.

 Assume $G(H)=\{ a_1, \dots, a_n\}$ and $S=K[x_1, \dots, x_n]$. Then $I_H$ is the binomial ideal
$$
I_H=\left(x^u-x^v:u,v\in \NN^n, \sum_{i=1}^n u_i a_i= \sum_{i=1}^n v_ia_i\right),
$$
where   for $u=(u_1, \dots, u_n)$ we let $x^u=x_1^{u_1}\cdots x_n^{u_n}$.
 The toric  ideal $I_H$  can be computed via elimination in the extended polynomial ring $S[t]$:
$$
I_H=(x_1-t^{a_1}, \dots, x_n-t^{a_n})S[t]\cap S,
$$
or via saturation, as follows.
Let $u_{(1)}, \dots, u_{(n-1)}$ be a $\ZZ$-basis for the subgroup $\{ (u_1, \dots, u_n)\in \ZZ^n : \sum_{i=1}^n u_ia_i =0\}$.
Then $I_H$ is the saturation  $$I_H=I_L: (x_1\cdots x_n)^\infty$$ of the lattice ideal 
$$
I_L=\left(x^{u^+_{(i)}}-x^{u^{-}_{(i)}}: i=1, \dots, n-1 \right).
$$
Here, for any vector $u$,  by   $u^+$ and $u^-$ we denote  the unique vectors with nonnegative  entries having disjoint support 
such that $u=u^+ - u^-$. We refer to  \cite[Chapters 4, 12]{Sturmfels} 
 for detailed proofs and further algorithms.

 When $\embdim(H) =2$, $I_H$ is  a principal ideal. If $\embdim(H) =3$, Herzog \cite{He-semi} showed that  $I_H$ can be generated by at most $3$ binomials, see also \cite{Denham}. For each $1\leq i\leq 3$, we look at the smallest positive multiple $c_ia_i$ which is in the semigroup generated by the other two generators of $H$, and this gives a binomial generator for $I_H$.  For instance, when $H=\langle 6,7,10 \rangle$ we may write 
\begin{eqnarray*}
4\cdot \phantom{1} 6 &=& 2\cdot 7+1\cdot 10, \\
4\cdot \phantom{1} 7 &=& 3\cdot 6+1\cdot 10, \\
2\cdot 10 &=& 1\cdot 6+ 2\cdot \phantom{1}7,
\end{eqnarray*}  and this gives 
\begin{equation}
\label{eq:3}
I_H=(x^4-y^2z, y^4-x^3z, z^2-xy^2).
\end{equation} 

In higher embedding dimension, it is more difficult in general to establish a system of generators for $I_H$ without using specialized software.

The maps in the resolution $\mathbb{F}$ are homogeneous with respect to the  grading induced by the semigroup, and 
this is reflected in the decomposition of the Betti numbers as sum of their multigraded parts:
$$
\beta_i(K[H])=\sum_{\lambda \in H} \beta_{i,\lambda} (K[H]).
$$

These summands can be expressed in terms of combinatorial and topological data. 
In \cite{BH-semi} and in \cite{CM} for any $\lambda$ in $H$, the  squarefree divisor complex $\Delta_\lambda$
 is  defined as the simplicial complex on the vertex set $\{1, \dots, n\}$ where $\{i_1, \dots, i_r\}$
 is a face of $\Delta_\lambda$  if $\lambda-a_{i_1}-\dots -a_{i_r} \in H$.
Here is a way to make use of its reduced homology groups.

\begin{Theorem}(Bruns-Herzog \cite[Proposition 1.1]{BH-semi}, Campillo-Marijuan \cite[Theorem 1.2]{CM})
\label{thm: betti-simplicial} In the above notation,
$$\beta_{i, \lambda} (K[H])= \dim_K \widetilde{H}_{i-1}(\Delta_\lambda; K)$$  for all $i>0$ and $\lambda$ in $H$.
\end{Theorem}
 
The reduced homology modules of $\Delta_\lambda$ may depend on the characteristic of the field $K$, hence the same is true 
for the Betti numbers of $K[H]$. However, the number of connected components of $\Delta_\lambda$, which is given by
$\dim_K H_0(\Delta_\lambda; K)$ does not depend on $K$ and thus $\beta_1(K[H])=\mu(I_H)$ depends only on $H$. 
We refer to \cite{BH-semi} for  more results on this direction.

\medskip

The tangent cone of $K[H]$ (or of $H$)   is the associated graded ring  of $K[H]$ with respect to the maximal ideal $\mm=(t^h:  h\in H\setminus\{0\})K[H]$, namely
$$
\gr_\mm K[H]= K[H]/\mm \oplus \mm/\mm^2 \oplus \mm^2/\mm^3 \oplus \cdots.
$$
  It is a standard graded $K$-algebra by letting $\mm^i/\mm^{i+1}$ be its homogeneous component of degree $i$. This grading is the one we shall further use for $\gr_\mm K[H]$, unless otherwise specified.

For any nonzero $f\in S$, its initial form $f^*$ is the homogeneous component  (with respect to the standard grading)  of smallest degree.
 For any ideal $I$ in $S$ we denote $I^*=(f^*: 0\neq f\in I)$ the ideal of initial forms.
If $f_1, \dots, f_r \in I$ such that  $I^*=(f_1^*, \dots, f_r^*)$ one says that $\{ f_1, \dots, f_r\}$ is a standard basis of $I$.
Moreover, in that situation the polynomials $f_1, \dots, f_r$ generate $I$.

The ideal $I^*$ is obtained from  a set of generators $I=(f_1, \dots, f_r)$ as follows. Let $F_i$ be the homogenization of $f_i$ with respect to a new variable $x_0$, for $i=1, \dots, r$, and assume $G_1, \dots, G_s$ is a Gr\"obner basis for the ideal $(F_1, \dots, F_r)\subset S[x_0]$ with respect to a term order that refines the partial order by degree in $x_0$. If we set $g_i=G_i(1, x_1, \dots, x_n)$ for $i=1, \dots, s$ then 
$I^*_H= (g_1, \dots, g_s)$, see \cite[Proposition 15.28]{Eis} or \cite[Proposition 3.25]{EH} for a proof.
 
The relevance to us of this construction stems from the fact that $\gr_\mm K[H] \cong  S/I_H^*$.

One can verify that for $H=\langle 6,7,10 \rangle$ the three generators listed in \eqref{eq:3} are a standard basis, hence $I_H^*=(y^2z, y^4-x^3z, z^2)$.

\medskip

General deformation arguments (see \cite{Eis}) prove that 
$$
\beta_i(K[H]) \leq \beta_i(\gr_\mm K[H]) \text{ for all } i,
$$
which shows that when searching for upper bounds for the Betti numbers of $K[H]$, we may refer to the Betti numbers of its tangent cone. In practice, the latter are easier to compute due to the standard grading on $I_H^*$.
When $K[H]$ and $\gr_\mm K[H]$ have the same Betti sequence, one says that $K[H]$ (or $H$)  is of homogeneous type, following  the terminology of \cite{HRV}. In \cite[Theorem 3.17]{JZ} the authors give a sufficient condition for $H$ to be of homogeneous type.

Both algebras $K[H]$ and $\gr_\mm K[H]$  have Krull dimension one, but while the former is a domain, the latter is not reduced and 
$\depth \gr_\mm K[H] \leq 1$. By the Auslander-Buchsbaum formula,
$$\embdim(H)-1 \leq \projdim \gr_\mm K[H] \leq \embdim(H).$$  
The case  $\embdim(H)-1 = \projdim \gr_\mm K[H]$ is equivalent to  $\depth \gr_\mm K[H]=1$ (hence, by definition, $\gr_\mm K[H]$ is Cohen-Macaulay), i.e. there exists a regular element of positive degree in $\gr_\mm K[H]$. In our setting, this is equivalent 
to   $\widehat{t^{\mult(H)}}$ being regular on $\gr_\mm K[H]$.

Quite a bit of work (e.g. \cite{Katsabekis}, \cite{AKN}, \cite{Arslan}, \cite{AMS}, \cite{He-reg}, \cite{RV}, \cite{Garcia}, \cite{BarF}, \cite{HS}) 
was  directed towards finding  criteria to test if $\gr_\mm K[H]$ is Cohen-Macaulay for an arbitrary numerical semigroup, 
partly motivated by the fact that  in that situation the Hilbert function can be computed easier and it is non-decreasing.

The following hierarchy of rings is known: complete intersection $\implies$ Gorenstein  $\implies$ Cohen-Macaulay. We explain what these conditions mean for our algebras of interest.

The algebra $K[H]$ or $\gr_\mm K[H]$ is a complete intersection (CI for short) if its defining ideal can be generated by the minimum number of  polynomials allowed by Krull' s theorem, namely by $\height I_H=\height I_H^*= \embdim H-1$ elements. 
The CI property for  $K[H]$ and $\gr_\mm K[H]$ does not depend on the field $K$, but on some arithmetic  conditions among the generators of $H$.
Delorme \cite{Delorme} proved that $K[H]$  is CI if and only if the generators of $H$ can be obtained recursively via a process which is nowadays called gluing. D'Anna, Micale and Sammartano  \cite[Theorem 3.6]{DAMS} characterized the CI property for $\gr_\mm K[H]$ using the Ap\'ery set of $H$. 

One says that $K[H]$  or $\gr_\mm K[H]$  is a Gorenstein ring if it is Cohen-Macaulay and  its type equals one.
The Cohen-Macaulay condition always holds for $K[H]$.  
 Kunz \cite{Kunz} showed that $K[H]$ is Gorenstein  if and only if the semigroup $H$ is symmetric, i.e. for all $x\in \ZZ$ either $x\in H$ or $F(H)-x \in H$.
The Gorenstein property for $\gr_\mm K[H]$ was characterized by Bryant in \cite[Theorem 3.14]{Bryant}.
One special feature of Gorenstein algebras is that their Betti sequence is symmetric, see \cite[Theorem 3.3.7 (a), Corollary 3.3.9]{BH}.

\section{The type of a numerical semigroup}
 \label{sec:type}

The  Ap\'ery set  of the semigroup $H$ with respect to a nonzero integer $a$ in $H$ is  
$$
\Ap(H,a)=  \{ x\in H: x-a \notin H\}.
$$
 Clearly, its elements give different remainders modulo $a$ and   $|\Ap(H,a)|=a$ . The Ap\'ery set of $H$ is defined as $\Ap(H, \mult(H))$.

The pseudo-Frobenius numbers of   $H$ are the elements in
$$
PF(H)=\{ x\in \ZZ \setminus H: x+h\in H \text{ for all } h\in H, h>0 \}.
$$

Fr\"oberg, Gottlieb and H\"aggkvist in \cite{FGH} define the  type of $H$ by $\type (H)=|PF(H)|$.

In the following we explain why this purely arithmetic invariant equals the type of the semigroup ring $K[H]$, where $K$ is any field.
We recall some algebraic terminology.

The { type of  a Cohen-Macaulay local ring   $(R, \mm)$ is  $\type(R)=\dim_K \Ext_R^d(K, R)$, where $K=R/\mm$ and $d$ is the Krull dimension of $R$ .
Moreover, it $x_1, \dots, x_d \subset \mm$ is an $R$-regular sequence  then $\type(R)=\dim_K \Hom_R(K, R/(x_1, \dots, x_d))$, 
see \cite[Lemma 1.2.19]{BH}.
In case $R=A/I$ with $A$ a regular local ring and $I$ an ideal in $A$, then $\type(R)=\dim_K \Tor_d^A(K,R)$, 
see \cite[Lemma 3.5]{AoyamaGoto}. This result means that $\type(R)$ is the rank of the last nonzero module in the minimal free resolution of $R$ over $A$. 

The type of a Cohen-Macaulay ring $R$ is defined as the maximum of $\type(R_\mathfrak{p})$, 
where $\mathfrak{p}$ ranges in the set of maximal ideals of $R$.

\begin{Theorem}
\label{thm:sametype}
Let $K$ be any field and $H$ be any numerical semigroup. Then $\type(H)=\type(K[H])$.
\end{Theorem}

\begin{proof}
The ring $R=K[H]$ is positively graded by setting $\deg t^h=h$ for all $h$ in $H$. 
$R$ has a unique maximal graded ideal $\mm$, hence by \cite[Theorem pp. 75]{AoyamaGoto} we get that $\type(K[H])=\type(K[H]_\mm)$. 
The ring map $K[H]_\mm \to K[|H|]$ is flat and its fiber is the field $K$, hence  it follows from \cite[Proposition 1.2.16.(b)]{BH} that $\type(K[H]_\mm)=\type(K[|H|])$.

Let $h\in H$, $h>0$.  Since $t^h$ is a regular element on $K[|H|]$ we have that 
$\type(K[|H|])=\dim_K \Hom_{K[|H|]}(K, K[|H|]/(t^h) )$. 
It is an easy exercise to check that a $K$-basis for $\Hom_{K[|H|]}(K, K[|H|]/(t^h) )$ is given by the residue classes $\widehat{t^x}$
where $x$ ranges in the set 
$$
B=\{x\in \Ap(H,h) : x \neq h \text{ and } x+g\in h+H \text{ for all } g\in H\setminus{0} \}. 
$$
We also leave it to the reader to check the equality of sets $B=h+ PF(H)$. Therefore, 
$$
\type(H)=|PF(H)|=|B|=\type(K[|H|])= \type(K[H]),
$$
which finishes the proof.
\end{proof}

From the arithmetic definition of the type of $H$ one gets the following inequalities.

\begin{Theorem} (Fr\"oberg, Gottlieb, H\"aggkvist \cite{FGH}) Let $H$ be a numerical semigroup. Then
\begin{enumerate}
\item [(i)] $\type(H)< \mult(H)$;
\item [(ii)] if $\embdim(H) \leq 3$ then $\type(H)\leq 2$;
\item [(iii)]  $(\type(H)+1) \cdot n(H) \geq F(H)+1$, where $n(H)$ is the number of elements in $H$ which are smaller that its Frobenius number $F(H)$.
\end{enumerate}
\end{Theorem}

By a result of Kunz \cite{Kunz} a semigroup is symmetric if and only if its type equals $1$.
When $H$ is pseudosymmetric, its type  equals 2, but the converse is not true. 
For instance  $H=\langle 5,6,7 \rangle$   has $PF(H)=\{8,9\}$.

In embedding dimension at least $4$ there is no absolute upper bound on the type of the semigroup, as the following examples show; see also Section \ref{sec-edim4}.  Historically,  Backelin was the first to produce 4-generated semigroups whose type is arbitrarily large, see the next example.

\begin{Example} (Backelin, \cite[Example  pp. 75]{FGH})
{\em
Given the integers $n\geq 2$ and $r\geq 3n+2$, set $s=r(3n+2)+3$ and
$$
H=\langle s, s+3, s+3n+1, s+3n+2 \rangle.
$$
It is proven in \cite{FGH} that $\type(H)\geq 2n+2$. However, it is wrongly claimed in  \cite{FGH} that $\type(H)=2n+3$. Computations with Singular (\cite{Sing}) and GAP (\cite{GAP}, \cite{Num-semigroup}) indicate that 
both $K[H]$ and $\gr_\mm K[H]$ have the Betti sequence $(1, 3n+4, 6n+5, 3n+2)$.
}
\end{Example}
 
For the next example the type was computed  by Cavaliere and Niesi in \cite{CN}. They also show that for these semigroups
 the associated projective monomial curve  is Cohen-Macaulay, which is not the case for Bresinsky's semigroups discussed in detail  in Section \ref{sec:bres-semi}.
\begin{Example}
\label{ex:cavalieri-niesi}
 (\cite[(3.4)]{CN})
{\em For $a\geq 3$ let 
$$
H_a=\langle a^2-a, a^2-a+1, a^2-1, a^2 \rangle.
$$
Then $\type(H_a)=2a-4$. Moreover, based on computations with Singular  \cite{Sing}  we claim that the Betti sequences for $K[H_a]$ and $\gr_\mm K[H_a]$ are the same $(1,2a-2, 4a-7, 2a-4)$. 
%% actually, the homogenization of $ H_a$ has the same Betti sequence.
}
\end{Example}

Combining  Example \ref{ex:cavalieri-niesi} and  Proposition \ref{prop:betti-arslan} we see that any positive integer  appears as the type of a $4$-generated numerical semigroup. By the technique of gluing, 
such examples may be constructed  for any  embedding dimension larger than four, see Corollary \ref{cor:type-gluing}.

When $H$ is generated by an (almost) arithmetic sequence, there are sharp bounds for the type in terms of the embedding dimension, see also Eq. \eqref{eq:betti}. 

\begin{Proposition}
\begin{enumerate}
\item [(i)] (Tripathi, \cite{Tripathi})  If $H$ is generated by an arithmetic sequence then $\type(H)\leq \embdim(H) -1$.
\item [(ii)] (Garc\'ia Marco, Ram\'irez Alfons\'in, R\o dseth, \cite[Theorem 3.1]{GRR}) If $H$ is generated by an almost arithmetic sequence then
 $\type(H) \leq 2(\embdim(H)-2
)$.
\end{enumerate}
\end{Proposition}
 
% That these are sharp it follows from the formulas \ref{eq:betti} for $H$ generated by an arithmetic sequence.
% According to \cite{GRR}, when $H=\langle 3k+2, 3k+3, \dots, 4k+2, 5k+3 \rangle$ we have $t(H)=2k$.
 
As noted before, symmetric semigroups have type $1$. It is a natural question if for {\em closely related} classes of semigroups the type is at least bounded by (a linear function in) the embedding dimension.

Almost symmetric semigroups have been introduced by Barucci and Fr\"oberg  \cite{BF} as a class of semigroups close to the symmetric ones. The semigroup $H$ is almost symmetric if for every $x \in \ZZ\setminus H$ such that $F(H)-x \notin H$ we have
$\{x, F(H)-x \} \subseteq PF(H)$; see also \cite{Nari} for equivalent characterizations.

Answering a question of Numata in \cite{Numata}, Moscariello \cite{Moscariello} proves the following result, see also \cite{HeWa}.

\begin{Theorem}(\cite[Theorem 1]{Moscariello}) 
If $H$ is almost symmetric and $\embdim(H) =4$ then $\type(H) \leq 3$, which is a sharp bound. 
\end{Theorem} 

In embedding dimension larger than $4$ there are examples of almost symmetric semigroups with $\type(H) \geq \embdim(H)$.
Starting from such a semigroup,    Strazzanti in \cite[Remark 2.6.3, Example 2.6.4]{Strazzanti-thesis} constructs an  almost symmetric semigroup $L$  of  higher embedding dimension with $\type(L)- \embdim(L) > \type(H)-\embdim(H)$. This shows that there is no constant $c$ such that $\type(H) \leq \embdim(H)+c$ for every almost symmetric semigroup $H$.

 Extending  a question asked by Moscariello in \cite{Moscariello} we ask the following.
\begin{Question}
{\em
Is there any bound  for the type and the rest of the Betti numbers for $K[H]$ in terms of $\embdim(H)$ when $K[H]$  is almost/nearly Gorenstein?
}
\end{Question}

\begin{comment}
% this section is to be expanded at a later point in time
\section{Multiplicity}

Values/Bounds induced by the multiplicity

Extra structures: Gore/CI/almost Gore

MED, MEDSY papers by Rosales, irreducible semigps-Pacific Math J 2003
\end{comment}

\section{Shifted families of semigroups and  upper bounds for the number of  defining equations}
\label{sec:shifts}

A  recent idea used in the study of  Betti numbers of  semigroup rings was to examine their behaviour in families of semigroups. Firstly, for any sequence of nonnegative integers 
  $\ab: a_1 < \dots <  a_n$ 
it will be convenient to denote in this section by
 $I(\ab)$ the kernel of the $K$-algebra homomorphism 
$\varphi:K[x_1,\dots, x_n]\to K[t]$ letting $\varphi(x_i)=t^{a_i}$ for $i=1,\dots, n$.
For  any integer $k$ we set $\ab+k: a_1+k, \dots, a_n+k$ and we call $\{ \langle \ab+k \rangle \}_{k\geq 0}$ the
  shifted family of semigroups associated to the sequence $\ab$. 
Herzog and Srinivasan conjectured that the Betti numbers for the semigroup rings in this shifted family are eventually periodic in $k$, for $k \gg 0$. After partial results in \cite{GSS}, \cite{Marzullo} and \cite{JS}, that conjecture was proved by Vu \cite{Vu}  in the following generality.

\begin{Theorem}
\label{thm:vu} (Vu, \cite[Theorem 1.1]{Vu})
Let $\ab= a_1 <\dots <a_n$. There exists $k_V$ such that 
$$
\beta_i(I(\ab+k))= \beta_i(I(\ab+k+(a_n-a_1)))
$$ 
for all $i$ and all $k>k_V$.
\end{Theorem}

We fix $\ab$ and  denote $H_k=\langle \ab+k \rangle$.
A key ingredient in Vu's proof is the following fact.

\begin{Proposition}
\label{prop:uv}
(\cite[Corollary 3.7]{Vu}) For $k\gg 0$ the nonhomogeneous binomials in a minimal 
binomial system of generators  for $I(\ab+k)$ are of the form 
$x_1^{u_1} x_2^{u_2}\dots x_{n-1}^{u_{n-1}}- x_2^{v_2}\dots x_{n-1}^{v_{n-1}}x_n^{v_n}$  with $u_i v_i=0$ for $i=2,\dots, n-1$, $u_1, v_n >0$ and $\sum_{i=1}^{n-1} u_i > \sum_{i=2}^n v_i$.
\end{Proposition}

 With the translation $k\mapsto k+ (a_n-a_1)$, what changes in these binomials are the exponents of $x_1$ and $x_n$ which increase by a quantity that  is also periodic with period $a_n-a_1$.
 
Proposition \ref{prop:uv} was used by Herzog and Stamate (\cite{HS}) to derive that:

\begin{Theorem} 
\label{thm:hs}
(\cite[Theorem 1.4]{HS})
For $k > k_V$ the ring $\gr_\mm K[H_k]$ is Cohen-Macaulay and  it has the same Betti sequence as $K[H_k]$.   
\end{Theorem}

In \cite[Corollary 6.5]{JZ} Jafari and Zarzuela reprove this fact noticing that for $k\gg 0$ the semigroup $H_k$ is
 homogeneous,  a concept they introduce. 

As a corollary of these periodicity results, we see that in any of the families
$\{K[H_k]\}_{k\geq 0}$ or $\{\gr_\mm K[H_k]\}_{k\geq 0}$ properties like Gorenstein or complete intersection 
occur either eventually with period $a_n-a_1$, or only for a finite set of shifts $k$.

The bound $k_V$ given in \cite{Vu} is usually not  optimal, and it involves the Castelnuovo-Mumford regularity of the ideal $J(\ab)$ of homogeneous polynomials in $I(\ab+k)$ for some (and hence for all) $k$.

In case $n=3$, these statements may be sharpened.  Note that by the results in Section \ref{sec:edim3}, the  periodicity for the Betti sequence of $K[H_k]$  now  means that for $k\gg 0$ the semigroup $H_k$ is symmetric (actually CI) with period $a_3-a_1$. The principal (i.e. the smallest)  period might be smaller, and it is determined in  \cite[Theorem 3.1]{S-3semi} in terms of the sequence $\ab$.
A smaller value than $k_V$ for the shift $k$ where periodicity occurs  is provided in \cite{S-3semi}. 
Also, in that paper  exact formulas are conjectured  for the thresholds from where the Betti numbers of $K[H_k]$, respectively of $\gr_\mm K[H_k]$ start changing periodically.

\medskip
Another corollary of Theorems \ref{thm:vu} and \ref{thm:hs} is that if we bound the width of the semigroup $H$, the Betti numbers of $K[H]$ and $\gr_\mm K[H]$ are bounded, as well.  
We recall that the width of the  numerical semigroup $H$ is the difference between the largest and the smallest minimal generator of $H$.

Two statements have been formulated  in \cite{HS} regarding the number of generators for $I_H^*$ (hence also for $I_H$). 
Firstly, for a numerical semigroup $H$ minimally generated by $a_1<\dots <a_n$ we define 
its   interval completion  to be the semigroup $\widetilde{H}$ generated by all the integers in the interval $[a_1, a_n]$.

\begin{Conjecture} (\cite[Conjectures 2.1, 2.4]{HS})
\label{conj:one} For any numerical semigroup $H$ one has
\begin{enumerate}
\item [(i)] $\mu(I_H^*) \leq {{\width(H) +1 }\choose 2 }$,
\item [(ii)] $\mu(I^*_H) \leq \mu(I^*_{\widetilde{H}})$. 
\end{enumerate}
\end{Conjecture}

 If correct, the second point of the conjecture above  would imply the first one. 
Also,  since $\widetilde{H}$ is generated by an arithmetic sequence, we may use Eqs.~\eqref{eq:betti} and   obtain effective bounds for $\mu(I_H^*)$.

Betti numbers for intersections of toric ideals  (and of their ideals of initial forms) of numerical semigroups in the same shifted family are considered in \cite{CS}.  Given the sequence $\ab$ as above,  let 
$$
\mathcal{I}_\mathcal{A}(\ab)=\bigcap_{k\in \mathcal{A}} I(\ab+k) \text{ and } \mathcal{J}_\mathcal{A}(\ab)=\bigcap_{k\in \mathcal{A}} I(\ab+k)^*.
$$ 
It is conjectured in \cite{CS} that  their Betti numbers are preserved under shifting $\mathcal{A} \mapsto \mathcal{A}+(a_n-a_1)$,  if $\min \mathcal{A} \gg 0$; see \cite[Proposition 2.3]{CS}  for some proved cases.

Conaway et al.  in \cite{ConawayAll}, and O'Neill and Pelayo in \cite{ONeill-Pelayo-Apery} study factorization invariants, respectively the Ap\'ery sets  in shifted semigroups $H_k$ for $k\gg 0$.
As a byproduct they obtain the following.

\begin{Proposition} 
\label{prop:better}
(\cite[Theorems 3.4, 4.9]{ConawayAll}, \cite[Theorem 4.9]{ONeill-Pelayo-Apery}) For  $k > (a_n-a_1)^2-a_1$
\begin{enumerate}
\item[(i)] the number of minimal generators for $I(\ab+k)$ and the type of $  H_k $ are periodic in $k$ with period $a_n-a_1$;
\item[(ii)] the conclusion of Proposition \ref{prop:uv} holds. 
 \end{enumerate}
\end{Proposition}
 
With the same argument as in \cite{HS},   Proposition \ref{prop:better}(ii) implies that $\gr_\mm K[H_k]$ is Cohen-Macaulay and  it has the same Betti sequence as $K[H_k]$ already for $k > (a_n-a_1)^2-a_1$, improving Theorem \ref{thm:hs}.

\section{Betti numbers for simple gluings}
\label{sec:betti-gluings}

Having its roots in  describing the structure of CI semigroups (\cite{Delorme}, \cite{Watanabe}, \cite{He-semi}),
the technique of gluing   has been used  to create  interesting  examples. We refer to Delorme's \cite{Delorme},  Rosales' \cite{Rosales} 
and the monograph \cite{NS-book} for a more general definition.

For our purposes, we consider a special case. Let $L$ be a numerical semigroup minimally generated by $a_1<\dots <a_n$, and $c>1$ and $d$ 
coprime integers such that $d\in L\setminus \{a_1, \dots, a_n\}$.
Then 
$$
H=\langle c L, d\rangle= \langle c a_1, \dots, c a_n, d\rangle
$$
 is said (in \cite{HS-koszulcone}) to be obtained from $L$ by a simple gluing.
%\footnote{[MOVE]In the terminology of \cite[Definition 3.1]{AM}, such an $H$ is called an extension of $L$.}

The order of $d$ in $L$ is defined as $\ord_L(d)=\max \{\sum_{i=1}^{n} \lambda_i: d=\sum_{i=1}^n \lambda_i a_i, \lambda_i  \in \NN \}$. 

If we  write $d=\lambda_1 a_1+\dots +\lambda_n a_n$ with $\lambda_i$ nonnegative integers and $\sum_{i=1}^n \lambda_i$ maximal,  and we consider the gluing relation
$$
f= x_{n+1}^c- \prod_{i=1}^n x_i^{\lambda_i},
$$ 
then  it is known (e.g. from the proof of Lemma 1 in \cite{Watanabe}, or \cite[Theorem 1.4]{Rosales})  
that 
$$
I_H= (I_L, f).
$$

The following lemma is the key step in describing the Betti numbers of $K[H]$ and $\gr_\mm K[H]$ in terms of the Betti numbers of $K[L]$ and $\gr_\mm K[L]$, respectively.

\begin{Lemma} 
\label{lemma:regular}
(\cite[Theorem 2.3, Lemma 2.6]{HS-koszulcone})
With notation as above, the following hold:
\begin{enumerate}
\item[(i)] $f$ is regular on $K[x_1,\dots, x_{n+1}]/I_L$,
\item[(ii)] if $c\leq \ord_L d$, then $f^*$ is regular on $K[x_1,\dots, x_{n+1}]/I_L^*$ and $I_H^*=(I_L^*, f^*)$.
\end{enumerate}
\end{Lemma}

Here is the main result of this section.

\begin{Theorem}
\label{thm:betti-gluing}
Let $L$ be a numerical semigroup minimally generated by $a_1<\dots <a_n$, and $c>1$ and $d$ coprime integers with $d\in L \setminus\{a_1, \dots, a_n\}$. 
Denote $H=\langle cL, d\rangle$. 
\\ The following hold:
\begin{enumerate}
\item[(i)] $\beta_i( K[H])= \beta_i(K[L]) + \beta_{i-1}(K[L])$, for all $i>0$.
\item[(ii)] if $c\leq \ord_L(d)$, then $\beta_i( \gr_\mm K[H])=  \beta_i(\gr_\mm K[L])+ \beta_{i-1}(\gr_\mm K[L])$, for all $i>0$. 
\end{enumerate}
\end{Theorem}

\begin{proof} Both statements are a consequence of K\"unneth's formula  
 for the homology of the tensor product of two complexes, combined with Lemma \ref{lemma:regular}.
We give a detailed proof of part (i), the other one is proved similarly.

Denote $R=S[x_{n+1}]=K[x_1, \dots, x_{n+1}]$.
Let $\mathcal{L}$ be a minimal free $S$-resolution of $S/I_L$.  
Tensoring this over $S$ by $R$ we get the complex $\mathcal{P}$ which is a minimal free $R$-resolution of $R/I_LR$.
We consider the complex $\mathcal{Q}: 0\to R \stackrel{f}{\rightarrow} R \to 0$. 
Clearly $H_i(\mathcal{Q})= R/(f)$ if $i=0$, and $0$ otherwise.
Also, $\mathcal{Q}_n$ is free (and flat) for any $n$, and  its image through the differential is either $0$, if $n\neq 0$, or $(f)R \cong R$, which is a free (and flat) $R$-module. 

From K\"unneth's formula (\cite[Theorem 3.6.3]{Weibel}), for any $n$ there exists a short exact sequence
$$
0\to \bigoplus_{i+j=n}H_i(\mathcal{P})\otimes H_j(\mathcal{Q}) 
\to H_n (\mathcal{P}\otimes\mathcal{Q})
\to \bigoplus_{i+j=n-1} \Tor_1^R(H_i(\mathcal{P}), H_j(\mathcal{Q}))) \to 0.
$$ 

For $n=0$ this gives $H_0(\mathcal{P}\otimes\mathcal{Q}) \cong  H_0(\mathcal{P}) \otimes H_0(\mathcal{Q})\cong R/(I_L, f)$, while for $n>1$ the previous exact sequence  gives $H_n(\mathcal{P}\otimes\mathcal{Q})=0$.

Plugging in $n=1$ we get the exact sequence 
$$
0 \to H_1(\mathcal{Q}) \to H_1(\mathcal{P}\otimes\mathcal{Q}) \to \Tor_1^R(H_0(\mathcal{P}), H_0(\mathcal{Q})) \to 0,
$$
hence  $ H_1(\mathcal{P}\otimes\mathcal{Q}) \cong \Tor_1^R(R/I_LR, R/(f))$. 
Note that by Lemma \ref{lemma:regular}, $f$ is regular on $R$ and on $R/I_L R$,
hence by \cite[Proposition 1.1.5]{BH}, the complex obtained by tensoring with $R/(f)$ the free resolution $\mathcal{P}$ of $R/I_LR$, is again exact. 
Therefore, $\Tor_1^R(R/I_LR, R/(f))=0$ and $\mathcal{P}\otimes\mathcal{Q}$ is exact and  a minimal 
 free resolution of $R/(I_L, f)$.
 
Clearly, for $n>0$ one has $(\mathcal{P}\otimes\mathcal{Q})_n= (P_n\otimes Q_0) \oplus (P_{n-1}\otimes Q_1) \cong P_n \oplus P_{n-1}$.
This shows that $\beta_i( K[H])= \beta_i(K[L]) + \beta_{i-1}(K[L])$ for all $i>0$.
\end{proof}

An immediate consequence of part (i) is the following result of Fr\"oberg, Gottlieb and H\"aggkvist \cite{FGH}.
\begin{Corollary}
\label{cor:type-gluing}
(\cite[Proposition 8]{FGH}) If $H=\langle cL, d \rangle$ is a simple gluing, then $K[H]$ and $K[L]$, and also $H$ and $L$ have the same type.
\end{Corollary}

\begin{proof} Denoting $n=\embdim(L)$, since $K[H]$ is a Cohen-Macaulay $S$-module of dimension $1$,  
by Theorem \ref{thm:betti-gluing}(i) the type of $K[H]$ is $\beta_{n}(K[H])= \beta_{n}(K[L])+ \beta_{n-1}(K[L])= \beta_{n-1}(K[L])$, hence $K[H]$ and $K[L]$ have the same type. The rest follows  from Theorem \ref{thm:sametype}.
\end{proof}

\begin{Corollary}
\label{cor:cm-gluing} 
If $H=\langle cL, d \rangle$ is a simple gluing and $c\leq \ord_L(d)$, then $\gr_\mm K[H]$ is Cohen-Macaulay or Gorenstein if and only if  $\gr_\mm K[L]$ is Cohen-Macaulay, respectively Gorenstein.
\end{Corollary}

\begin{proof} Let $n=\embdim(L)$. By the Auslander-Buchsbaum formula \cite[Theorem 1.3.3]{BH}, 
$\gr_\mm K[H]$ and  $\gr_\mm K[L]$ are Cohen-Macaulay if and only if $\beta_{n}(\gr_\mm K[H])=0$, respectively
$\beta_{n+1}(\gr_\mm K[L])=0$.  The conclusion follows from Theorem \ref{thm:betti-gluing}(ii) and Corollary \ref{cor:type-gluing}.
\end{proof}

\begin{Remark} 
{\em The condition $c\leq \ord_L(d)$ in Corollary \ref{cor:cm-gluing} can not be removed. Indeed, for $h>1$ let $L= \langle 3h, 3h+1, 6h+3 \rangle$ 
and $H= \langle (h^2+1)L, 3h^2 \rangle$ from Example \ref{ex:shibuta}. Clearly, $\ord_L(3h^2)=h < h^2+1$.
As noted in Examples \ref{ex:shibuta-type}  and \ref{ex:shibuta} respectively, 
the tangent cone $\gr_\mm K[L]$ is not Cohen-Macaulay, however $\gr_\mm K[H]$ is  Cohen-Macaulay.
}
\end{Remark}

\begin{Remark}
{\em
A simple gluing $H=\langle cL, d \rangle$ with $c \leq \ord_L(d)$ is called a nice extension in \cite{AM}. 
This concept is generalized in \cite[Definition 2.1]{AMS} by Arslan, Mete and \c{S}ahin  to nice gluings of arbitrary numerical semigroups.
In \cite[Theorem 2.6]{AMS} the authors show that a nice gluing of two semigroups $H_1$ and $H_2$ with Cohen-Macaulay tangent cones, also has a Cohen-Macaulay tangent cone.
This gives another proof for the Corollary \ref{cor:cm-gluing} above.
\c{S}ahin \cite{Sahin} studies nice extensions for arbitrary toric varieties, and he formulates the results in this section in that generality. 

The Betti numbers of $K[H]$ when $H$ is a numerical semigroup obtained by gluing have also been considered in  \cite{GimSri} 
and \cite{Numata-gluing}.
}
\end{Remark}

\section{Some examples}
\label{sec:firstexamples}

For some classes of numerical semigroups  the   resolution of $K[H]$ and of $\gr_\mm K[H]$ is known.

\subsection{Complete intersections}\label{subsec:ci}
When $K[H]$ is a complete intersection ring, its defining ideal is generated by $a=\embdim(H)-1$ binomials which are a regular sequence in $S$, 
hence the Koszul complex on these binomials provides a  minimal $S$-free resolution of $S/I_H$, see \cite[Corollary 1.6.14]{BH}. Consequently, the Betti sequence for
$K[H]$ is
$$
\left(1, {a \choose 1}, {a \choose 2}, \dots, {a \choose a} \right).
$$
The numerical semigroups $H$ for which $K[H]$ is  CI have been characterized combinatorially by Delorme \cite{Delorme} in terms of gluings, see also Section~\ref{sec:betti-gluings}.
For instance, when $H$ is generated by a geometric sequence, the rings $\gr_\mm K[H]$ and   $K[H]$ are CI, see \cite[Proposition 3.4]{HS-koszulcone}.

\subsection{Arithmetic sequences}
\label{sec:arithmetic}
When $H=\langle a, a+d, \dots, a+ (n-1)d \rangle$ is generated by an arithmetic sequence with $\gcd(a,d)=1$, $n\leq a$,  
Gimenez, Sengupta and Srinivasan in \cite{GSS} show that if we denote by  $b$ the unique integer such that $a \equiv b \mod (n-1)$ and $1\leq b \leq n-1$, then
\begin{eqnarray}
\label{eq:betti}
\beta_i(K[H])= i { n-1 \choose i+1} +
\begin{cases}
(n- b-i +1) {n-1 \choose  i-1}  &\text{ if }  1 \leq i \leq n-b, \\
(i-n+ b) {n-1 \choose i} & \text{ if } n-b < i \leq n-1,
\end{cases}
\end{eqnarray}
see also the preprint \cite{OT-as} of Oneto and Tamone for an independent, yet similar approach.
 
The same values for the Betti numbers of $\gr_\mm K[H]$ had been obtained by Sharifan and Zaare-Nahandi in \cite{Sha-Za-1}, 
and the equality $\beta_i(K[H])=\beta_i(\gr_\mm K[H])$, for all $i$, was noted by them in \cite{Sha-Za-2}.
Independently, in \cite{HS}  using just the relations defining $K[H]$,  Herzog and the author   also showed that $\gr_\mm K[H]$ has the same Betti numbers as $K[H]$.

Formulas similar to \eqref{eq:betti} hold when $H$ is generated by a generalized arithmetic sequence,
 i.e. $H=\langle a, ha+d, ha+2d, \dots, ha+(n-1)d \rangle$ for some positive integers $h,d$ with $\gcd(a,d)=1$.
Details, and also the explicit minimal free resolution of $K[H]$ may be found in  \cite{Sha-Za-1}, \cite{GSS} and \cite{OT}.

\section{Embedding dimension at most $3$}
\label{sec-atmost3}

\subsection{The $2$-generated case.}
 If $H=\langle a_1, a_2 \rangle$ with $a_1<a_2$ and $\gcd(a_1, a_2)=1$, then $K[H]\cong K[x,y]/(x^{a_2}-y^{a_1})$ and $\gr_\mm K[H] \cong K[x,y]/(y^{a_1})$, hence both algebras have the Betti sequence $(1,1)$.

\subsection{The $3$-generated case.} 
\label{sec:edim3}
If $\embdim(H)=3$, by Herzog's work in \cite{He-semi} one has that $\mu(I_H) \leq 3$, hence the possible Betti sequences of $K[H]$ are $(1,2,1)$ if it is CI, or $(1,3,2)$ if it is not.
Moreover, Herzog \cite{He-reg} and Robbiano-Valla \cite{RV} show that $\gr_\mm K[H]$ is Cohen-Macaulay if and only if $\mu(I_H^*) \leq 3$, and therefore, under such an extra hypothesis the
possible Betti sequences of $\gr_\mm K[H]$ are again $(1,3,2)$ and $(1,2,1)$.

On the other hand, there are examples of Shibuta (\cite[Example 5.5]{Goto}) of $3$-generated semigroups $H$ where $\mu(I_H^*)$ is arbitrarily  large. 
Extending that family, one has the following
\begin{Example} {\em (\cite[Section 3.3]{HS})
\label{ex:shibuta-type}
   For $a> 3$ and 
\begin{equation*}
\label{eq:bigtgcone}
H_a= \langle a,a+1, 2a+3 \rangle 
\end{equation*}
 one obtains that $\mu(I_{H_a}^*)= \lfloor \frac{a-1}{3} \rfloor +3 $, by explicitly computing the ideal of initial forms 
$$
I_{H_a}^*= (xz, z^{k+1}) + y^{\varepsilon} \cdot(y^3, z)^k.
$$ 
Here we wrote $a=3k+\varepsilon$ with $k=\lfloor \frac{a-1}{3} \rfloor$  and $1 \leq \varepsilon \leq 3$.
Numerical experiments with Singular \cite{Sing} indicate that, in this notation,  the whole Betti sequence of 
$\gr_\mm K[H_a]$  is  $(1, k+3, 2k+2, k)$.
}
\end{Example}

\section{Large Betti numbers in embedding dimension $4$}
\label{sec-edim4}

When $\embdim(H)\geq 4$, there is no upper bound depending on $\embdim(H)$ alone for the number of defining equations and the rest of the Betti numbers of $K[H]$, as the following examples show. 

\subsection{Bresinsky semigroups}
\label{sec:bres-semi}
  
For  $h\geq 2$,  Bresinsky \cite{Bres} considered the numerical semigroup 
\begin{equation*}
\label{eq:semi}
B_h=\langle (2h-1)2h, (2h-1)(2h+1), 2h(2h+1), 2h(2h+1)+2h-1\rangle.
\end{equation*}
and he showed that the number of defining equations for $K[B_h]$ is  at least $2h$. We now compute its whole Betti sequence.

\begin{Theorem}
\label{thm:main-brez}
With notation as above, both algebras $K[B_h]$ and $\gr_\mm K[B_h]$ have  the Betti sequence 
$$(1,4h, 8h-4, 4h-3).
$$
\end{Theorem}

\begin{proof} 
We fix an integer $h\geq 2$ and for brevity we denote $I=I_{B_h}\subset S=K[x,y,z,t]$.
Based on computations started in \cite{Bres}, it is proven in \cite[Section 3.3]{HS} that 
$$
 \{xt-yz\} \cup  \{ z^{i-1}t^{2h-i}-y^{2h-i}x^{i+1}: 1\leq i \leq 2h \} \cup \{x^{2h+1-j}z^j-y^{2h-j}t^j: 0\leq j \leq 2h-2\}
$$
is a minimal generating set and a minimal standard basis of $I$, and that $\gr_\mm K[B_h]$ is Cohen-Macaulay.
%The latter implies that $x$ is a regular element on $S/I^*$. 

We let $\bar{S}=K[y,z,t]$, $\bar{I}$ be the canonical projection of $I\subset S$ onto $\bar{S}$, and similarly for $\bar{I^*}$. 
It is immediate to check that
\begin{equation*}
\label{eq:ibar2}
\bar{I}= \bar{I^*} = (yz) +  ( t^{2h-1}, zt^{2h-2}, \dots, z^{2h-1}) +  ( y^2t^{2h-2},  y^3t^{2h-3},  \dots ,  y^{2h-1}t  , y^{2h}).
\end{equation*}
Clearly $x$ is regular on $S/I$ which is a domain, but also on $S/I^*$ since $\gr_\mm K[B_h]$ is Cohen-Macaulay, as noted above.
Therefore, the Betti numbers for  $S/I$, $S/I^*$,  and $\bar{S}/\bar{I}$ coincide. The conclusion follows from the next lemma. 
\end{proof}

\begin{Lemma}
\label{lemma:ibar}
Consider the ideal $J \subset \bar{S}=K[y,z,t]$ defined as 
\begin{equation}
\label{eq:ibar}
J=(yz)+(t^h, zt^{h-1}, \dots, z^h)+(y^2 t^{h-1}, y^3 t^{h-2}, \dots, y^{h+1}).
\end{equation}
The Betti sequence of $\bar{S}/J$ is $(1, 2h+2, 4h, 2h-1)$.
\end{Lemma}

\begin{proof}
Clearly $\beta_0(\bar{S}/J)=1$ and $\beta_1(\bar{S}/J)= \mu(J)=2h+2$.	

Since $\bar{S}/J$ has finite length, it is Cohen-Macaulay. Hence its projective dimension as an $\bar{S}$-module equals  $3$ 
and its last nonzero Betti number satisfies $\beta_3(\bar{S}/J)= \dim_K \Soc(\bar{S}/J)$.
Note that 
$$ 
\Soc(\bar{S}/J)=\left\{ \widehat{f} \in \bar{S}/J: (y,z,t)\cdot \widehat{f} =\hat{0} \right\}
$$ 
has a $K$-basis consisting of the residue classes of the monomials $m \notin J$ such that
$(y,z,t)\cdot m \subset J$. 

Let $m=y^a z^b t^c $ be such a monomial. Clearly $ab=0$.

If $a=0$, then $b>0$, otherwise $m=y^c \notin J$ implies $c\leq h$, and together with $ym=y t^c \in J$ we get a contradiction.
Moreover, $m=z^bt^c \notin J$ yields $b+c \leq h-1$. 
Since $zm\in J$, one gets that $b+c+1 \geq h$, hence $b+c=h-1$ and $m\in \mathcal{S}_1= \{ zt^{h-2}, z^2t^{h-3}, \dots, z^{h-1} \}$.

If $a>0$, then $b=0$. Since $m=y^at^c \notin J$ we get $c<h$. Also, $ym=y^{a+1}t^c \in J$ implies $a+1+c \geq h+1$, i.e. $a+c \geq h$.
From $tm=y^at^{c+1} \in J$ we get that either $a=1$ and $c+1=h$, or $a\geq 2$ and $a+c+1 = h+1$.  
We obtain that $m\in \mathcal{S}_2= \{ yt^{h-1}, y^2t^{h-2}, \dots, y^h  \}$.
It is easy to check that  $\widehat{m} \in \Soc(\bar{S}/J)$ for all $m\in \mathcal{S}_1\cup \mathcal{S}_2 $, therefore $\beta_3(\bar{S})= 2h-1$.

From the relation $\sum_i \beta_i(\bar{S}/J)=0$ we see that $\beta_2(\bar{S})= 4h$, and this finishes the computation of the Betti sequence.
\end{proof}

\subsection{Arslan semigroups}

In \cite[Proposition 3.2]{Arslan} Arslan shows that for the family of semigroups
\begin{equation}
\label{eq:arslan}
A_h= \langle h(h+1), h(h+1)+1, (h+1)^2, (h+1)^2+1 \rangle, \text{ where } h\geq 2, 
\end{equation}
 the defining ideal of $K[A_h]$ is
\begin{multline}
\label{eq:eq-arslan}
I_{A_h}= ( x^{h-i}z^{i+1}-y^{h-i+1}t^i: 0\leq i < h) + \\ 
			 (z^i t^{h-i}-x^{i+1}y^{h-i}: 0\leq i \leq h)+  (xt-yz).
\end{multline}
and  later on he proves that $\gr_\mm K[A_h]$ is Cohen-Macaulay via some considerations involving Gr\"obner bases.
\begin{comment}
In the original paper \cite{Arslan} there is one more (non-minimal) equation listed for the generators of $I_{A_h}$, 
which is of the form in the first bracket above,  for $i=h$. Namely,
$$z^{h+1}-yt^h=z (z^h-x^{h+1})+x (x^h z-y^{h+1})-y (t^h-xy^h).$$
\end{comment}

With notation as before, going modulo $x$ we obtain
\begin{equation*}
\bar{I}_{A_h}= (y^{h-i+1}t^i: 0\leq i < h)+ (z^i t^{h-i} : 0\leq i \leq h) + (yz).
\end{equation*}
This is a monomial ideal whose generators, naturally forming a standard basis,
may be clearly lifted to binomials in $I_{A_h}$ with the same initial degree.
According to Herzog's criterion \cite[Theorem 1]{He-reg}, as formulated in \cite[Lemma 1.2]{HS}, 
we get that the generators in \eqref{eq:eq-arslan} for $I_{A_h}$ form a standard basis as well,
and that $x$ is regular on $\gr_\mm K[A_h]$. Hence we reobtain  Arslan's result in  \cite[Proposition 3.4]{Arslan}:
$$
I_{A_h}^*=(x^{h-i}z^{i+1}-y^{h-i+1}t^i: 0\leq i <h) +  
			 (z^i t^{h-i} : 0\leq i \leq h)+  (xt-yz).
$$

Modulo $x$, the latter ideal is the same as $\bar{I}_{A_h}$ and it coincides with the ideal $J$ in \eqref{eq:ibar}. 
Using  Lemma \ref{lemma:ibar}, we derive the next result.

\begin{Proposition}
\label{prop:betti-arslan}
Let $h\geq 2$ and $A_h$ be the Arslan semigroup defined in \eqref{eq:arslan}. 
The semigroup ring $K[A_h]$ and its tangent cone $\gr_\mm K[A_h]$ have the same Betti sequence $(1,2h+2, 4h, 2h-1)$.
\end{Proposition}

These computations show that in the class of $4$-generated numerical semigroups, 
even among those with Cohen-Macaulay tangent cone,
 the Betti numbers of $K[H]$ may be arbitrarily large.
Using the gluing construction described in Section \ref{sec:betti-gluings}, 
we can exhibit examples of semigroups with arbitrarily large Betti numbers in any  higher embedding dimension.

\section{Embedding dimension 4, continued}
\label{sec:4continued}
\subsection{AA-sequences}
A sequence of integers is called an  almost arithmetic (AA) sequence  if it consists of an arithmetic sequence and of one more element.
Any $3$-generated semigroup is generated by an AA-sequence.

 Kumar~Roy, Sengupta and Tripathi \cite{KST} described the minimal resolution of $K[H]$ when $H$ is minimally generated by an AA-sequence with $4$ elements.
Similar to the results in Section \ref{sec:arithmetic}, they obtain that only the following eight  Betti sequences are possible:
 $(1,3,3,1)$, $(1,4,5,2)$, $(1,4,6,3)$, $ (1,5,5,1)$, $(1,5,6,2)$, $(1,5,7,3)$, $(1,6,8,3)$, $(1,6,9,4)$.

\subsection{Symmetric semigroups}
Assume  $H$ is a $4$-generated symmetric  semigroup. 
 If $H$ is CI, the Betti sequence of $K[H]$ is $(1,3,3,1)$, as seen in Section \ref{subsec:ci}. 

When $H$ is not CI, Bresinsky \cite{Bres-gore} described its generators and he explicitly computed  the defining relations for $K[H]$.
Based on that parametrization,   Barucci, Fr\"oberg and \c{S}ahin \cite{BFS} described the minimal free resolution of $K[H]$.
The Betti sequence is always $(1,5,5,1)$ in that case.

\begin{Remark}
{\em
Micale and Olteanu \cite{MO} notice that  in embedding dimension at least five, more than one Betti sequence is possible for $K[H]$ when $H$ is symmetric and not CI. Indeed, if $a$ and $d$ are coprime positive integers such that $a \equiv 2 \mod (4)$, 
letting $H=\langle a, a+d, a+2d, a+3d, a+4d \rangle$, it follows from Eq.~\eqref{eq:betti} that the Betti sequence of $K[H]$ is $(1,9,16,9,1)$.

On the other hand, the semigroup $H=\langle 12,20, 28,30, 35\rangle$ is symmetric, not CI and 
the Betti sequence of $K[H]$ is $(1,6,10,6,1)$. A large class of symmetric semigroups with this Betti sequence is provided by \cite[Proposition 5.1]{MO}. 
}
\end{Remark}

When $H$ is $4$-generated, symmetric and not CI, it is a natural question to find the equations of the tangent cone,   since we know the five equations defining $K[H]$.
Despite the effort (see \cite{Katsabekis}, \cite{AM}), explicit formulas are not   available in all cases.

Arslan, Katsabekis and Nalbandiyan \cite{AKN}  gave  necessary and sufficient conditions for a $4$-generated symmetric and not CI
  semigroup to have a Cohen-Macaulay tangent cone, in terms of Bresinsky's parametrization; see also \cite[Theorem 2.4]{Katsabekis} for a more compact formulation of their result.

Recently, under these hypotheses for $H$, building on the results in  \cite{AKN}, Katsabekis \cite{Katsabekis} 
shows that in several cases when $\gr_\mm K[H]$ is Cohen-Macaulay, then $\mu(I_H^*) \in \{5,6\}$ 
by explicitly finding a standard basis for $I_H$.
It is likely that this program can be completed to find the possible Betti sequences of the tangent cone, at least in those several cases. We ask if that is a finite list or not. We also wonder if  $\mu(I_H^*)$ can be determined in all the cases when $\gr_\mm K[H]$ 
is Cohen-Macaulay, compare with Example \ref{ex:shibuta}.

Here are some examples. 

\begin{Example}
{\em 
We used Singular \cite{Sing} to   compute  the Betti sequences of $\gr_\mm K[H]$ for some   $4$-generated symmetric and non-CI
  semigroups $H$ given in \cite[Examples 3.14, 3.21, 3.28, 4.3]{AKN}: 
\begin{enumerate}
\item[(a)]
for $H$ any of  $\langle 1199, 2051, 2352, 3032 \rangle$, $\langle 627, 1546, 1662, 3377\rangle$, or \\  $\langle 813, 1032, 1240, 1835\rangle$  
		 the tangent cone $\gr_\mm K[H]$ is Cohen-Macaulay and its Betti sequence is $(1,6,8,3)$;
\item[(b)]
 for $H=\langle 2m+1, 2m+3, 2m^2+m-2, 2m^2+m-1 \rangle$ with $m\geq 4$,  the ideal $I_H^*$ is explicitly computed %in \cite[Example 4.3]{AKN}, 
  and one gets that
       $\gr_\mm K[H]$ is not Cohen-Macaulay, and its Betti sequence is $(1,8,14,9,2)$.
\end{enumerate}
}
\end{Example}

\begin{Example}
{\em
Arslan and Mete (see \cite[Example 2.1]{AM}) show that for $m\geq 2$ the semigroup  
$$
H_m= \langle m^3+m^2-m, m^3+2m^2+m-1, m^3+3m^2+2m-2, m^3+4m^2+3m-2\rangle
$$ 
is symmetric and not CI, and that $\gr_\mm K[H_m]$ is Cohen-Macaulay by explicitly computing 
$$
I_{H_m}^*=(zt^{m-1}, y^{m+2}, z^m, t^m, y^mt^{m-1}).
$$
Denoting $J=I_{H_m}^* \mod x$, it is easy to check that  $\{ y^{m-1}z^{m-1}t^{m-2} , y^{m-1}t^{m-1} \}$ is a $K$-basis for $\Soc(\bar{S}/J)$, 
hence arguing as in the proof of Lemma \ref{lemma:ibar} we get that $\gr_\mm K[H_m]$ 
has the Betti sequence $(1,5,6,2)$.
}
\end{Example}

\medskip

We next consider the possible number of equations for $\gr_\mm K[H]$ when $H$ is CI and $\embdim(H)=4$.

Recall  that in  embedding dimension $3$, $\gr_\mm K[H]$ is CM if and only if $\mu(I_H^*) \leq 3$.
Working to extend this result, Shibuta \cite[Theorem 3.1]{Shibuta} shows   that if $H$ is a $4$-generated CI numerical semigroup 
and $\mu(I_H^*) \leq 4$, 
then $\gr_\mm K[H]$ is Cohen-Macaulay.  The paper \cite{Shibuta}  is also the source of the following example.

\begin{Example} 
\label{ex:shibuta}
{\em (Shibuta, \cite[Proposition 3.14]{Shibuta}) For $h\geq 2$  the semigroup
\begin{equation}
H_h=\langle 3h^2, 3h(h^2+1), (3h+1)(h^2+1), (6h+3)(h^2+1) \rangle
\end{equation}
is CI with $I_{H_h}=(y^h-x^{h^2+1}, yt-z^3, t^h-x^{2(h^2+1)}y) \subset S=K[x,y,z,t]$ and $I_{H_h}^*=(yt)+(z^3,y)^h+(z^3, t)^h$. 

Clearly $x$ is regular in $S/I_{H_h}^*$ and $\gr_\mm K[H_h]$ is Cohen-Macaulay. 
We shall find its Betti numbers as in the proof of Lemma \ref{lemma:ibar}.
Letting $J= I_{H_h}^* \mod x$ and $\bar{S}=K[y,z,t]$, it is easy to check that 
$$
\{ y^{h-1}z^{3i-1}: 1\leq i \leq h\} \cup \{ z^{3i-1}t: 1\leq i \leq h-1 \} 
$$ is 
a $K$-basis for $\Soc(\bar{S}/J)$. Therefore the Betti sequences for $\gr_\mm K[H_h]$ and for $\bar{S}/J$ are $(1, 2h+2, 4h, 2h-2)$.
}
\end{Example}

\begin{Example}
\label{ex:katsabekis}
{\em (Katsabekis,  \cite[Example 3.6]{Katsabekis}) For $m\geq 1$ the semigroup
$$
H_m=\langle 8m^2+6, 8m^2+10, 12m^2+15, 20m^2+15 \rangle
$$
is CI, with $I_{H_m}=(x^5-t^2, y^3-z^2, x^{2m^2}t-y^{2m^2}z)$. Using the methods described in Section \ref{sec:algebra}
one gets that  $I_{H_m}^*= ( z^2, t^2,      x^{2m^2}t-  y^{2m^2}z, y^{2m^2+3}t, y^{4m^2+3})$,  and the tangent cone is not Cohen-Macaulay. 
%% since  (I_{H_m}^*: x)= I_{H_m}^* + (x^{2m^2-1}zt) \neq  I_{H_m}^*.
Testing with Singular (\cite{Sing}) for small values of $m$ we always get the  Betti sequence  $(1,5,9,6,1)$ for $\gr_\mm K[H_m]$.
}
\end{Example}

Starting with a $3$-generated  CI   semigroup $H$ whose tangent cone has large Betti numbers, by a simple gluing (see Section \ref{sec:betti-gluings})
 we can obtain a $4$-generated CI  whose tangent cone   has large Betti numbers, as well.

\begin{Example}
\label{ex:gluing-3ci}
{\em 
For $m\geq 2$, the semigroup $H_m= \langle 3m, 3m+1, 6m+3\rangle$ in Example \ref{ex:shibuta-type} is CI and the Betti sequence for $\gr_\mm K[H_m]$ is $(1, m+2, 2m, m-1)$.
Letting 
$$
L_m= \langle 2H_m, 6m+1 \rangle =\langle 6m, 6m+2, 12m+6, 6m+1 \rangle,
$$
we have that $L_m$ is CI. By Theorem \ref{thm:betti-gluing}, the tangent cone $\gr_\mm K[L_m]$, which  is not Cohen-Macaulay, 
has the Betti sequence $(1, m+3, 3m+2, 3m-1, m-1)$.
}
\end{Example}

\begin{Remark}
{\em
Following Examples \ref{ex:shibuta} and \ref{ex:gluing-3ci}, 
we expect that arbitrarily large Betti numbers may be obtained for $\gr_\mm K[H]$ also when $H$ is $4$-generated, symmetric and not CI. 
}
\end{Remark}

\subsection{Pseudosymmetric semigroups}
A semigroup $H$ is called pseudosymmetric if $F(H)$ is even and $PF(H)=\{F(H)/2, F(H) \}$.
 Komeda \cite{Komeda} characterized the generators of any $4$-generated { pseudosymmetric} numerical semigroup $H$ and he found the defining relations of $K[H]$.
Based on that,   Barucci, Fr\"oberg and \c{S}ahin \cite{BFS} described the minimal free resolution of $K[H]$. Consequently, its Betti sequence is $(1,5,6,2)$. 
 
The defining equations or the Betti numbers of the tangent cone are not known for this class of semigroups. 
\c{S}ahin and \c{S}ahin (\cite{SahinSahin}) describe the Cohen-Macaulay property of $\gr_\mm K[H]$ in terms of Komeda's parametrization.

\subsection{Further extensions}
Assume $H$ is a  $4$-generated  semigroup which is not CI.
Eto \cite{Eto} describes the almost symmetric such semigroups: their generators, the defining equations of $K[H]$ and also the minimal resolution of $K[H]$. When the type of $H$ is $2$, $H$ is pseudosymmetric and  this was discussed above.
Moscariello \cite{Moscariello}  had proven that otherwise $H$ must have type equal to $3$, and Eto shows that the possible Betti sequences are  $(1,6,8,3)$ and $(1,7,9,3)$.  This completes some very partial results   in \cite{Numata}.

The  class of  nearly Gorenstein  semigroups has been recently introduced in \cite{HHS} and it  contains the almost symmetric ones.
\begin{Problem}
{\em Find a parametrization of the $4$-generated nearly Gorenstein  semigroups and describe the minimal resolution of their semigroup ring.}
\end{Problem}

\begin{Problem}
{\em Describe the Betti numbers and the minimal resolution of $\gr_\mm K[H]$ when $H$ is a $4$-generated semigroup which is (almost) symmetric or nearly Gorenstein.}
\end{Problem}

\medskip

{\bf Acknowledgement}.
We gratefully acknowledge the use of the  Singular \cite{Sing} software and of the \texttt{numericalsgps} package \cite{Num-semigroup} in GAP \cite{GAP} for the development of this paper.
The   author was  supported  by a fellowship at the Research Institute of the University of Bucharest (ICUB). 
 
We thank  Ignacio Garc\'ia-Marco for sending corrections to an initial version of this paper,  to Francesco Strazzanti for useful pointers to the literature and to an anonymous referee  for suggestions that improved the presentation. A great moral debt is owed to J\"urgen Herzog since our joint projects served as an introduction to the topic of this survey and also as  a motivation to write it.

{}

\end{document}